\documentclass[11pt, letterpaper]{amsart}

\usepackage{color}
\usepackage{amsfonts}
\usepackage{amssymb}
\usepackage{amsmath}
\usepackage{hyperref}

\usepackage{graphicx}
\usepackage{enumitem}
\usepackage{subcaption}

\newcommand{\one}{\mathbf{1}}

\newcommand{\pd}{\partial}

\newcommand{\ch}{\operatorname{ch}}

\newcommand{\al}{\alpha}

\newcommand{\la}{\lambda}
\newcommand{\vf}{\varphi}
\newcommand{\om}{\omega}

\newcommand{\cF}{\mathcal F}

\newcommand{\cI}{\mathcal I}

\newcommand{\cT}{\mathcal T}

\newcommand{\bD}{\mathbb D}
\newcommand{\bE}{\mathbb E}

\newcommand{\bI}{\mathbb I}

\newcommand{\bR}{\mathbb R}

\newcommand{\bV}{\mathbb V}

\newtheorem{theorem}{Theorem}[section]

\newtheorem{lemma}[theorem]{Lemma}

\theoremstyle{definition}
\newtheorem{remark}[theorem]{Remark}
\newtheorem{defin}[theorem]{Definition}

\numberwithin{equation}{section}

\newcommand{\dif}{\mathop{}\!\mathrm{d}} % \mathop produces two thin spaces, \! removes the trailing one
\newtheorem{conj}[equation]{Conjecture}
\newcounter{vremennyj}

\begin{document}

\title[Multi-parameter embeddings for $p\neq 2$]%
{Multi-parameter Carleson embeddings\\ for $p\neq 2$ on $T^2$ or for $p=2$ on $T^4$ and why the proofs fail}

%\author[N.~Arcozzi]{Nicola Arcozzi}
%\address[N.~Arcozzi]{Universit\`{a} di Bologna, Department of Mathematics, Piazza di Porta S. Donato, 40126 Bologna (BO)}
%\email{nicola.arcozzi@unibo.it}
%\thanks{NA is partially  \text{supp}orted by the grants INDAM-GNAMPA 2017 ``Operatori e disuguaglianze integrali in spazi con simmetrie'' and PRIN 2018 ``Variet\`{a} reali e complesse: geometria, topologia e analisi armonica''}
%\author[I.~ Holmes]{Irina Holmes}
%\thanks{IH is partially  \text{supp}orted by the NSF an NSF Postdoc under Award No.1606270}
%\address{Department of Mathematics, Michigan Sate University, East Lansing, MI. 48823}
\author[P.~Mozolyako]{Pavel Mozolyako}
\thanks{PM is  \text{supp}orted by the Russian Science Foundation grant 17-11-01064}
\address[P.~Mozolyako]{St. Petersburg State University, St. Petersburg, Russia}
\email{pmzlcroak@gmail.com}
\author[G.~Psaromiligkos]{Georgios Psaromiligkos}
\address[G.~Psaromiligkos]{Department of Mathematics, Michigan Sate University, East Lansing, MI. 48823}
\email{psaromil@math.msu.edu}
\author[A.~Volberg]{Alexander Volberg}
\thanks{AV is partially  \text{supp}orted by the NSF grant DMS 1900268 and by Alexander von Humboldt foundation}
\address[A.~Volberg]{Department of Mathematics, Michigan Sate University, East Lansing, MI. 48823 and Hausdorff Center, Universit\"at Bonn}
\email{volberg@math.msu.edu}
%\author[P.~Zorin-Kranich]{Pavel Zorin-Kranich}
%\thanks{PZ was partially  \text{supp}orted by the Hausdorff Center for Mathematics (DFG EXC 2047)}
%\address[P.~Zorin-Kranich]{Mathematical Institute, University of Bonn, Bonn, Germany}
%\email{pzorin@uni-bonn.de}
% 42B	Harmonic analysis in several variables
% 42B20	Singular and oscillatory integrals (Calder?on-Zygmund, etc.)
% 42B35	Function spaces arising in harmonic analysis
% 47A	General theory of linear operators
% 47A30	Norms (inequalities, more than one norm, etc.)
%{30E20, 47B37, 47B40, 30D55.}
%
% 30D55	$H^p$-classes (1980-2009)
% 30E20	Integration, integrals of Cauchy type, integral representations of analytic functions
%
% 47B   	Special classes of linear operators
% 47B37	Operators on special spaces (weighted shifts, operators on sequence spaces, etc.)
% 47B40	Spectral operators, decomposable operators, well-bounded operators, etc.
\begin{abstract}
This note contains a plethora of counterexamples to attempts to generalize the results of bi-parameter embedding from $p=2$ case to either $p>2$ or $p<2$. This is in striking juxtaposition to $p=2$  case that was fully understood in the series of papers \cite{AMPS},   \cite{AMPVZ-K},  \cite{MPVZ1}, \cite{MPVZ2}, \cite{AHMV}, \cite{MPV}.   We also build a counterexample to small energy majorization on bi-tree. This counterexample shows that straightforward generalizations of methods of \cite{AMPVZ-K}, \cite{MPVZ1}, \cite{MPVZ2}, \cite{AHMV} from $2$-tree $T^2$ or $3$-tree $T^3$ to  $4$-tree $T^4$ will not work even for $p=2$ unless some new approach is invented.
\end{abstract}
\maketitle
\section{Introduction}
Embedding theorems on graphs are interesting in particular because they are related to the structure of spaces of holomorphic functions. For Dirichlet space on a disc  this fact has been explored in \cite{ARSW}
\cite{ArcoRS2002}
\cite{ARSW11}, and for Dirichlet space on bi-disc in \cite{AMPS}, \cite{AMPVZ-K}, \cite{AHMV}. Bi-disc case is much harder as the corresponding graph has cycles. One particular interesting case see in \cite{Saw1} (a small piece of bi-tree is considered).

The difference between one parameter theory (graph is a tree) and two parameter theory (graph is a bi-tree) is huge.
One explanation is that in a multi-parameter theory all the notions of singular integrals, para-products, BMO, Hardy classes etc become much more subtle than in one parameter  settings. There are many examples of this effect.  It was demonstrated in results of S.Y. A. Chang, R. Fefferman and L. Carleson, see \cite{Carleson}, \cite{Chang}, \cite{ChF1},\cite{RF1}, \cite{TaoCar}.

Another difference between one- and two-parameter embeddings is that in one parameter case the results for $L^p$ are the same as for $L^2$. This seems not to be the case for the two parameter theory.

\section{Basic lemmas that underpin $p=2$ case}
\label{basic}

We know that $d$-parameter Carleson embedding theorem is completely understood when $d=1$ and $1<p\le \infty$ (Carleson, Sawyer), and for $p=2$, $d=2,3$, see \cite{AMPVZ-K},  \cite{MPVZ1}, \cite{MPVZ2}, \cite{AHMV}, \cite{MPV}. The cases 1) $d=2$, $p\neq 2$, 2) $d=4$, $p=2$ are the simplest open cases. 

Here we demonstrate the difficulties in understanding those simplest open cases by building a plethora of counterexamples to key lemmas ``generalized'' to those cases.

We start by listing and proving the key  lemmas from \cite{AMPVZ-K} that underpin the main bi-parameter embedding result of this and subsequent papers. Then we show why they break down for $p\neq 2$ case.

In Section \ref{fg-cex} we come back to $p=2$ case but for $4$-parameter embedding. And we write a counterexample to a statement that would be one of the possible tools to crack the  $4$-parameter case.

All this does not mean that we have counterexamples to natural statements. Below we only have counterexamples to ``natural proofs''.

\begin{defin}
\label{def:superadditive}
Given a finite tree $T$, the set of \emph{children} of a vertex $\beta\in T$ consists of the maximal elements of $T$ that are strictly smaller than $\beta$:
\[
\ch \beta := \{\max ( \beta' \in T \colon \beta' < \beta)\}\,.
\]
A function $g : T \to [0,\infty)$ is called \emph{superadditive} if for every $\beta \in T$ we have
\[
g(\beta) \geq \sum_{\beta' \in \ch(\beta)} g(\beta').
\]
\end{defin}

\begin{lemma}
\label{lem:supadditive-l1linf}
Let $T'$ be a finite tree and $g,h : T' \to [0,\infty)$.
Assume that $g$ is superadditive and $Ih \leq \lambda$ on $\ \text{supp} g$.
Then for every $\beta \in T'$ we have
\[
\sum_{\alpha \leq \beta} g(\alpha) h(\alpha)
\leq
\lambda g(\beta).
\]
\end{lemma}
\begin{proof}
Without loss of generality we may consider the case when $\beta$ is the unique maximal element of $T'$ and $T' = \ \text{supp} g$.
We induct on the depth of the tree.
Let $T'$ be given and  \text{supp}ose that the claim is known for all its branches.
Then by the inductive hypothesis and superadditivity we have
\begin{align*}
\sum_{\alpha \leq \beta} g(\alpha) h(\alpha)
&=
g(\beta) h(\beta) + \sum_{\beta' \in \ch(\beta)} \sum_{\alpha \leq \beta'} g(\alpha) h(\alpha)
\\ &\leq
g(\beta) h(\beta) + \sum_{\beta' \in \ch(\beta)} g(\beta') \sup_{\alpha \leq \beta'} \sum_{\alpha \leq \alpha' \leq \beta'} h(\alpha')
\\ &\leq
g(\beta) h(\beta) + \sum_{\beta' \in \ch(\beta)} g(\beta') \sup_{\alpha < \beta} \sum_{\alpha \leq \alpha' < \beta} h(\alpha')
\\ &\leq
g(\beta) h(\beta) + g(\beta) \sup_{\alpha < \beta} \sum_{\alpha \leq \alpha' < \beta} h(\alpha')
\\ &=
g(\beta) \sup_{\alpha \leq \beta} \sum_{\alpha \leq \alpha' \leq \beta} h(\alpha').
\qedhere
\end{align*}
\end{proof}

\begin{lemma}
\label{lem:I2-positive}
Let $I$ be an integral operator with a positive kernel and $f,g$ positive functions.
Then
\[
\int (If)^{2} g \leq \Bigl( \sup_{\ \text{supp} g} II^{*}g \Bigr) \int f^{2}.
\]
\end{lemma}
\begin{proof}
Without loss of generality $f$ is positive.
By duality we have
\[
\int (If)^{2} g
=
\int f I^{*}(If \cdot g)
\leq
\|f\|_{2} \| I^{*}(If \cdot g)\|_{2}.
\]
By the hypothesis $Ih(x) = \int K(x,y) h(y)$ with a positive kernel $K$.
Hence
\begin{align*}
\|I^{*}(If \cdot g)\|_{2}^{2}
&=
\int I^{*}(If \cdot g) I^{*}(If \cdot g)
\\ &=
\int K(x,y) ((If)(x) g(x)) K(x',y) ((If)(x') g(x')) \dif(x,x',y)
\\ &\leq
\int \frac12 (If(x)^{2}+If(x')^{2}) K(x,y) (g(x)) K(x',y) (g(x')) \dif(x,x',y)
\\ &=
\frac12 \int I^{*}((If)^{2} \cdot g) I^{*}(g) + \int I^{*}(g) I^{*}((If)^{2} \cdot g)
\\ &=
\int (II^{*}g) \cdot (If)^{2} \cdot g
\\ &\leq
\Bigl( \sup_{\ \text{supp} g} II^{*}g \Bigr) \int (If)^{2} \cdot g.
\end{align*}
Substituting the second displayed estimate into the first we obtain
\[
\int (If)^{2} g
\leq
\|f\|_{2} \Bigl( \sup_{\ \text{supp} g} II^{*}g \Bigr)^{1/2} \Bigl( \int (If)^{2} \cdot g \Bigr)^{1/2}.
\]
The conclusion follows by rearranging the terms.
\end{proof}

\begin{lemma} 
\label{Phi}
Let $g,f,w : T \to [0,\infty)$ be positive functions and $\lambda,\delta>0$.
Assume that $g$ is superadditive and $I (wg) \leq \delta$ on $\ \text{supp} f$.
Then there exists a positive function $\phi : T \to [0,\infty)$ such that
\begin{equation} 
\label{geIf}
I (w\phi) \gtrsim I(wf) \text{ on } \{ \lambda/2 < I(wg) \leq 2\lambda\},
\end{equation}
and
\begin{equation} \label{enest1}
\int w \phi^{2} \lesssim \frac{\delta}{\lambda} \int w f^{2}.
\end{equation}
\end{lemma}

\begin{proof}
Without loss of generality we may assume $\lambda \geq 4\delta$.
Define
\begin{equation}
\label{PhiDef}
\phi(\alpha) :=
\frac{1}{\lambda} \one_{\delta < I(wg)(\alpha) \leq 2\lambda} I(wf)(\alpha) g(\alpha)
\end{equation}

We prove first \eqref{geIf}.
Let $\omega\in T$ be such that $\lambda/2 < I(wg)(\omega) \leq 2\lambda$.
Then for every $\alpha\in T$ with $\alpha\geq\omega$ and $I(wg)(\alpha)>\delta$ we have
\[
I(wf)(\alpha) = I(wf)(\omega).
\]
It follows that
\begin{align*}
I(w\phi)(\omega)
&=
\frac{1}{\lambda} \sum_{\substack{\alpha\geq\omega :\\ \delta < I(wg)(\alpha) \leq 2\lambda}} I(wf)(\alpha) (wg)(\alpha)
\\ &=
I(wf)(\omega) \frac{1}{\lambda} \sum_{\substack{\alpha\geq\omega :\\ \delta < I(wg)(\alpha)}} (wg)(\alpha)
\\ &=
I(wf)(\omega) \frac{1}{\lambda} ( I(wg)(\omega) - I(wg)(\alpha_{min}) ),
\end{align*}
where $\alpha_{min}$ is the smallest $\alpha$ outside of the summation range if it exists (otherwise that term is omitted).
But then $I(wg)(\omega) \geq \lambda/2$ and $I(wg)(\alpha_{min}) \leq \delta$, and \eqref{geIf} follows.

Next we will prove the energy estimate \eqref{enest1}.
Let $\mathcal{U} := \{I(wg) \leq \delta\}$, so that $\mathcal{U}$ is an up-set and $f$ is  \text{supp}orted on $\mathcal{U}$.
By Lemma~\ref{lem:I2-positive} with the operator $\cI:=I \sqrt{w} \one_{\mathcal{U}}$ and functions $F:=f \sqrt{w}$ and $G:=g^{2} w \one_{I(wg)\leq 2\lambda}$ we can estimate
\begin{align*}
\int w \phi^{2}
&\leq
\frac{1}{\lambda^{2}} \sum_{\substack{\alpha :\\ I(wg)(\alpha) \leq 2\lambda}} I(wf)(\alpha)^{2} g(\alpha)^{2} w(\alpha)
\\&=
\frac{1}{\lambda^{2}}\int \cI(F)(\al)^2 G(\al)\leq^{Lemma \ref{lem:I2-positive}}\frac{1}{\lambda^{2}}\Bigl( \sup_{\ \text{supp} G} \cI\cI^{*}G \Bigr) \int F^{2}=
\\ &\leq
\frac{1}{\lambda^{2}} \Bigl( \int w f^{2} \Bigr) \sup I(w \one_{\mathcal{U}} I^{*}(g^{2} w \one_{I(wg)\leq 2\lambda})).
\end{align*}
By Lemma~\ref{lem:supadditive-l1linf} with the superadditive function $g:=g \one_{I(wg) \leq 2\lambda}$ and the function $h:=wg$ we can estimate
\[
I^{*}(g^{2} w \one_{I(wg)\leq 2\lambda}) \leq 2\lambda g.
\]
Moreover, since $\mathcal{U}$ is an up-set on a simple tree we have
\begin{equation}
\label{one-dim}
I (w \one_{\mathcal{U}} g) \leq \sup_{\mathcal{U}} I(wg) \leq \delta.
\end{equation}
Combining the last three displays we obtain the energy estimate \eqref{enest1}.
\end{proof}

\bigskip

\subsection{The ultimate technical result}
\label{ult}
Below is the lemma that is the backbone of main results of \cite{AMPVZ-K},  \cite{MPVZ1}, \cite{MPVZ2}. 
It was proved using Lemma\ref{Phi}, where one puts
$$
f^\beta(\al):= F(\al, \beta)
$$
$$
g^\beta(\al):= I_2 F(\al, \beta)=\sum_{\gamma\ge \beta} F(\al, \gamma)\,.
$$
In its turn we just saw that Lemma \ref{Phi} is a combination of  Lemma \ref{lem:I2-positive} and Lemma \ref{lem:supadditive-l1linf}.

Here is this ``backbone lemma''.
\begin{lemma}
\label{SmEMaj2}
Let $T^{2}$ be a $2$-tree and $F:T^2\to  [0, \infty)$ a function that is superadditive in each parameter separately..
Suppose that $\text{supp}\, F \subseteq \{ \bI (F) \leq \delta \}$. Let $\la \ge 4 \delta$.
Then there exists $\varphi:T^2\to  [0, \infty)$ such that
\[
a)\, \,\bI \varphi \ge \bI F, \,\, \text{where}\,\, \bI F \in [\la, 2\la],
\]
\[
b) \int_{T^2}  \varphi^2 \le C\frac{\delta^2}{\la^2} \int_{T^2}  F^2,
\]
where $C$ is an absolute constant.
\end{lemma}

\begin{remark}
Below we show that the analog of the above Lemma \ref{SmEMaj2} for $p\neq 2$ hits an obstacle: one needs to prove the analog of Lemma \ref{Phi}. If one uses our approach to prove this lemma by combination of Lemma \ref{lem:I2-positive} and Lemma \ref{lem:supadditive-l1linf}, then one of them works for $p\le 2$, and another works for $p\ge 2$.
\end{remark}

\begin{remark}
Many statement written above are true on $T^2$ instead of $T$.  For example, Lemma \ref{lem:I2-positive} does not care about where it happens. However, the statement \eqref{one-dim} is blatantly wrong on $T^2$. There is no maximum principle on $T^2$. This is the reason why \cite{AMPVZ-K} required  not only Lemmas above, but considerably more work. Notice  that Lemma \ref{Phi} is just wrong on $T^2$. The counterexample is built in Section \ref{fg-cex}.
\end{remark}

\section{Trying generalizations for $p\neq 2$. And failing}
\label{positive}

We are going to prove the next theorem on a simple finite tree $T$.

\begin{theorem}
\label{inter}
Let $f,g$ be positive functions on $T$ and numbers $0<\delta<<\lambda$ with $ \text{supp}(f)\subseteq\big\{ I g\leq \delta\big\}$. If $g$ is super-additive and $I g\leq \lambda$ on $T$ then we have for $p\geq 2$:

$$||I f\cdot g||_{\ell^p(T)}\leq C \delta^\frac{p-1}{p}\cdot\lambda^{\frac{1}{p}} ||f||_{\ell^p(T)}$$
\end{theorem}

\begin{proof}

The case $p=2$ follows by using Lemma \ref{Phi} above (that is using two lemmata of \cite{AMPVZ-K}). %First, we use  Lemma 2.2 where we take $h=g$. Then, we use Lemma 2.3 where we set $g\rightarrow g^2$. An inspection of the proof (see section \ref{another} for details) shows that the supremum on Lemma 2.3 can be taken over $ \text{supp}(f)\cap\ \text{supp}(g)$ which gives the required $\delta$. 

We will prove the case $p=\infty$ and hence the theorem follows by interpolation.

The inequality we want to prove becomes

$$||I f\cdot g||_{\ell^{\infty}(T)}\leq C \delta\cdot ||f||_{\ell^{\infty}(T)}$$

We assume the tree $T$ is finite and hence the supremum is achieved on the left side. Without loss of generality we can assume its achieved for some $\beta\in \text{ \text{supp}} f$: Suppose that is achieved for some $v$, not necessarily in $\text{ \text{supp}} f$. Then let $\beta$ be the smallest ancestor of $v$ such that $f(\beta)\neq 0$. Obviously $I f(v)=I f(\beta)$ and hence:

$$I f(v)g(v)=I f(\beta)g(v)\leq I f(\beta)g(\beta)$$
as $g$ is non-decreasing. By construction, $$||I f\cdot g||_{\ell^{\infty}(T)}=I f(\beta)g(\beta)$$ and $\beta\in \text{ \text{supp}} f$. Suppose this $\beta$ is a descendant of the root of the $n$-th generation. Then,

$$I f(\beta)g(\beta)\leq ||f||_{\ell^{\infty}(T)}\cdot n\cdot g(\beta) \leq ||f||_{\ell^{\infty}(T)}\cdot I g(\beta) $$ using again that $g$ is non-decreasing. But since $\beta \in \text{ \text{supp}} f$ we have by assumption that $I g(\beta)\leq \delta$, and so the inequality follows.

\end{proof}

\begin{remark}
By repeating the main arguments in the above proof we can get the following:
\end{remark}

\begin{lemma}\label{infinity}
Let $g$ be a positive increasing function. Then for any positive function $f$:

$$\|I f\cdot g\|_{\ell^{\infty}}\leq \sup\limits_{\text{ \text{supp}}(g)\cap \text{ \text{supp}}(f)}\big(I I ^*g\big)\cdot\|f\|_{\ell^{\infty}}$$

\end{lemma}

\subsection{Counterexample to an attempt for $p<2$}

Now we claim that Theorem \ref{inter} is not true when $1<p<2$. Of course the constant $C$ in the above theorem should be independent of the depth of the simple tree $T$. If $1<p<2$ then there is some simple tree $T$ of depth $k+2^k$ ($k\in\mathbb{N}$ to be specified) where this theorem fails.

Hence, we start with a simple tree $T$ of depth $k+2^k$ and we name $R^j_i$ the $j$-th dyadic interval of the $i$-th generation, $0\leq i\leq k+2^k$, $1\leq j\leq 2^i$. With this, $R^1_0$ is the root. Now we start constructing the functions involved in the counterexample. 

The function $g$ equal to $1$ on the root, equal to $\frac{1}{2}$ on the whole first generation, $\frac{1}{4}$ on the whole second generation and so on. We define $g$ like this up to generation $k$ and at each node of the $k+1$-th generation we give the value $\frac{1}{2^k}$ on the left child and 0 on the right. Recursively we give $g$ the value $2^{-k}$ on the left child of every node where $g$ is $2^{-k}$ and the value $0$ on the right. If at some node $g$ is 0, we put $g=0$ on both children of this node. With this construction $g$ is super-additive.

Let's make a simple observation. At generation $k$ we have $2^k$ nodes where $g$ is non-zero. Each one of these gives exactly one node in the next generation where $g$ is non-zero (=$2^{-k}$) and so on. Hence, for every $i$ with $k+1\leq i\leq 2^k+k$ we have exactly $2^k$ nodes where $g$ is non-zero and equal to $2^{-k}$.

Hence, if we want to bound $I g$ then we just have to calculate it at the boundary, i.e at a small square of side $2^{-(k+2^k)}$ where $g$ is non-zero. Such a square is a descendant of the root of order $k+2^k$ and hence we have so many values of $g$ involved. Recall that $g(R_i^j)=2^{-i}$ for $0\leq i\leq k$ and since every $\om$ has only one ancestor $\al$ in each generation, in particular for $k+1\leq i\leq k+2^k$, we get $g(\al)\leq 2^{-k}$. We have then 

$$I g(\om)\leq\sum\limits_{\al\geq\om}g(\al)=\sum\limits_{i=0}^{k}2^{-i}+\sum\limits_{i=k+1}^{k+2^k} 2^{-k}=\sum\limits_{i=0}^{k}2^{-i}+1 \leq 3$$ and from that we can choose $\lambda:=3$.

Now we proceed to define the function $f$. For a dyadic interval $R_i^j$ as above we define $f(R_i^j)=2^{-i}$ if $g(R_i^j)\neq 0$ and $f=0$ when $g=0$. This is a simple definition in order to have $\ \text{supp} (f)\subseteq \big\{I g\leq \delta\big\}$ and again we can choose $\delta:=3$ as well. We can easily then observe that $I f(R_i^j)\geq f(R_0^1)=1$ for any dyadic interval $R_i^j$ and that 

$$||f||^p_{\ell^{p}(T)}=\sum\limits_{i=0}^{k+2^k}\sum\limits_{j=1}^{2^i}f(R_i^j)^p= \sum\limits_{i=0}^{k}\sum\limits_{j=1}^{2^i}f(R_i^j)^p+\sum\limits_{i=k+1}^{k+2^k}\sum\limits_{j=1}^{2^i}f(R_i^j)^p=$$

$$ \sum\limits_{i=0}^{k}\sum\limits_{j=1}^{2^i}2^{-ip}+ \sum\limits_{i=k+1}^{k+2^k}2^{k} 2^{-ip}$$ where in the last sum only $2^k$ terms survive given the definition of $f$. Both of these  sums are finite  independently of $k$ ( as $1<p<2$), but depend on $p$. Thus $0<||f||_{\ell^{p}(T)}\leq C_p$.

Finally, let us estimate $\sum\limits_{T}(I f)^pg^p$ from below (Recall $I f(R_i^j)\geq 1$ )

$$\sum\limits_{T}(I f)^pg^p \geq \sum\limits_{i=k+1}^{k+2^k}\sum\limits_{j=1}^{2^i}\big[I f(R_i^j)g(R_i^j)\big]^p\geq \sum\limits_{i=k+1}^{k+2^k} 2^k 2^{-pk}$$ as at each one of these generations, $g$ is non-zero and equal to $2^{-k}$ exactly on $2^k$ nodes. Now the last sum is equal to $2^{(2-p)k}$. Since $1<p<2$ the power is positive (for $k$ large enough) and this is where the counterexample comes to life. In fact, if the inequality were true we would get
$$
2^{(2-p)k} \le C 3^p C_p,
$$
which is impossible for large $k$ as $p<2$.

%Suppose that the inequality raised to the power $p$ were true, for some constant $C\neq 0$.  Then choose $k$ to be such $2^{(2-p)k}>2C\cdot  3^p\cdot C_p$, where $C_p$ is as above. Then, collecting all the information we have gathered 

%$$2C\delta^{p-1}\lambda ||f||_{\ell^p(T)}  \leq 2C\cdot3^p\cdot C_p  < 2^{(2-p)k} \leq \sum\limits_{T}(I f)^p g^p \leq$$

%$$\underset{\text{if ineq holds}}{\leq} C \delta^{p-1}\lambda ||f||_{\ell^p(T)} $$ which is a contradiction as $C\neq 0$.

\section{Why not ``increasing'' instead of ``super-additive''?}

Looking at the proof of  Theorem \ref{inter} a natural question appears. When treating the case $p=\infty$ we only used that $g$ is increasing, a consequence of being super-additive. Hence, we may ask ourselves if we can replace ``super-additive'' with ``increasing''. 

Where did we use super-additivity? We used it only in Lemma \ref{Phi}, which is Lemma $2.2$ of \cite{AMPVZ-K}. It is important to get the result for $p=2$ (we need to interpolate). Let's state the lemma first (we take $g=h$ as this is what we actually want)

\begin{lemma}
\label{2.2-1}
Let $T$ be a finite tree and $g: T \to [0,\infty)$.
Assume that $g$ is super-additive and $I  g \leq \lambda$ on $\ \text{supp} \,g$.
Then for every $\gamma \in T$ we have
\begin{equation}
\label{gamma}
\sum_{\alpha \leq \gamma} g^2(\alpha)
\leq
\lambda g(\gamma).
\end{equation}
\end{lemma}

This lemma is true. Now, the question is, can we replace super-additivity with increasing on this lemma? The answer is no, and we will construct a function $g$ which is \textbf{increasing, and strictly sub-additive} but such that \eqref{gamma} fails.

Let $N\in\mathbb{N}$ and $T=T_N$ be a finite dyadic tree of depth $N$. We construct the function $g$ by the following rule: $g(I_0)=1$ and for every dyadic interval $I$ and its children $I_{-}, I_{+}$ we have $g(I_{-})=\frac{g(I)}{2}$ and $g( I_{+}) =g(I)$. This function is increasing and also strictly sub-additive as $g(I_{-})+g( I_{+}) =\frac{g(I)}{2}+g(I)>g(I)$.

Let's note several things. First of all, we can choose $\lambda=N$ as the maximum value $I  g$ can take is $N$. To see this, let $\om$ be the right-most boundary point as we look at the base of tree $T$. On the set $\{\al: \al\geq \om \}$ the function $g$ is equivalent to $1$ by construction. Hence $I g(\om)=N$ and there is no other path giving a bigger value than this (on other paths $g$ takes a value $<1$ at least on one node). 
%Also, for $\gamma$ we will choose $I_0$.

Second, lets name the dyadic intervals in this way: $I_{1,1}=I_0$ and $I_{i,j}$ is the $j$-th dyadic interval (we enumerate from the left to the right as we look at the tree) of the $i$-th generation ( $1\leq i \leq N$ and $1\leq j \leq 2^{i-1}$). 

Now, in the $i+1$-th generation ($i\geq 2$) the sum of $g^2$ over all $j$ splits into two categories. The nodes where $g$ is half the value of  $g$ at its father and the nodes where $g$ is equal to the value of $g$ at its father. The two categories have the same amount of members and since the number of $j$ is $2^{i-1}$ in total, each category has $2^{i-2}$ members. Namely,

$$\sum\limits_{j=1}^{2^{i}}g^2(I_{i+1,j})= \frac{1}{2^2}\cdot\sum\limits_{j=1}^{2^{i-2}}g^2(I_{i,j})+ \sum\limits_{j=1}^{2^{i-2}}g^2(I_{i,j})=\frac{5}{4}\cdot \sum\limits_{j=1}^{2^{i-2}}g^2(I_{i,j})$$

Which means that the sum on the next generation is $\frac{5}{4} \times$ the sum on the previous generation. Using this formula recursively we get that 

$$\sum\limits_{\al\leq I_0}g^2(\al)=\sum\limits_{i=1}^N\sum\limits_{j=1}^{2^{i-1}}g^2(I_{i,j})= \sum\limits_{i=1}^N (\frac{5}{4}\big)^i \cdot g^2(I_{1,1})=4\Big(\big(\frac{5}{4}\big)^N -1\Big)$$ as $g(I_{1,1})=g(I_0)=1$.

For $\gamma$ in we will choose $I_0$ to check that \eqref{gamma} fails.  But since $\lambda=N$ we have $\lambda g(\gamma)= N g(I_0)=N$ and obviously 

$$4\Big(\big(\frac{5}{4}\big)^N -1\Big)>> N$$ which means

$$\sum\limits_{\al\leq I_0}g^2(\al)>>\lambda g(I_0)$$
\eqref{gamma} fails blatantly, and therefore we can not replace super-additive requirement with increasing and sub-additive requirement.

\subsection{The same holds for general $p>1$} Let $p>1$. We look at Lemma \ref{Phi}(Lemma 2.2 of \cite{AMPVZ-K}) when we set $g\rightarrow g^{p-1}$ and take $h=g$. Then that lemma becomes:

\begin{lemma}
\label{2.2-2}
Let $T$ be a finite tree and $g: T \to [0,\infty)$.
Assume that $g^{p-1}$ is super-additive and $I  g \leq \lambda$ on $\ \text{supp} g$.
Then for every $\gamma \in T$ we have
\begin{equation}
\label{gammap}
\sum_{\alpha \leq \gamma} g^p(\alpha)
\leq
\lambda g^{p-1}(\gamma).
\end{equation}
\end{lemma}
 
Of course it holds, it is just a particular case of Lemma \ref{Phi}.

Now the question is again the same: Can we instead of ``$g^{p-1}$ super-additive'' have ``$g^{p-1}$ increasing''? Of course the latter is equivalent to $g$ increasing. The answer is still ``no'' and the above counterexample is the one which gives this answer.

Indeed, take the same $g$ as above. It is increasing as we have shown. Also $g^{p-1}$ is not super-additive, but it is strictly sub-additive as before: $g^{p-1}(I_+) +g^{p-1}(I_-) = \frac{1}{2^{p-1}}g^{p-1}(I)+ g^{p-1}(I)> g^{p-1}(I) $.

Now the sum $\sum\limits_{j=1}^{2^{i-1}}g^p(I_{i+i,j})$ is $\frac{2^p+1}{2^p} \times$ the sum of the previous generation and again for the same choices of $\lambda$ and $I_0$ we get 

$$\sum\limits_{\al\leq I_0}g^p(\al)= \Big(2^p+1\Big)\bigg(\Big(\frac{2^p+1}{2^p}\Big)^N - 1\bigg)$$

which is, again, much bigger than $N=\lambda g(I_0)$.

\subsection{Straightforward counterexample to Theorem \ref{inter} in case of increasing.} Our approach above was to find a counterexample to Lemma \ref{2.2-2} if we were to substitute $g$ ``super-additive'' with $g$ ``increasing''. 
Another approach is to go directly to Theorem \ref{inter} and replace super-additive with increasing. The same counterexample as above shows that this can not be done for \textbf{any} $p\geq 1$.

The setting of the counterexample is this. Let $g$ as above. Then, as before, we can take $\lambda:=N$. Now we can take $\delta:=1$, as we will define the function $f$ to be equal to $1$ and  supported on the set $\{\al:\al\geq \om_1\}$ where $\om_1$ is the left-most boundary point as we look at the base of the tree. We know that $g(I_0)=1/2$ and on its left child is half of it, on the left child of this child is half of it and so on. So $Ig(\om_1)\leq 1$ and for any $\al$ in the set above $Ig(\al)\leq Ig(\om_1)\leq 1$. 

Also, as $f$ is equal to $1$ on this set, $f(I_0)=1$. Therefore, for any $\al \in T$ we have $If(\al)\geq f(I_0)=1$. Additionally we have $\sum\limits_T f^p= N\cdot 1^p=N$.  But then we have 

$$\sum\limits_{T}(If\cdot g)^p \geq \sum\limits_{T} g^p \underset{above}{=} \Big(2^p+1\Big)\bigg(\Big(\frac{2^p+1}{2^p}\Big)^N - 1\bigg) >> N^2 = \delta^{p-1} \cdot \lambda \cdot \sum\limits_T f^p$$

\section{Try to prove Lemma 2.3 of \cite{AMPVZ-K} for $p> 2$}
\label{another}

\subsection{We can take supremum over  \text{supp}(f) in Lemma \ref{lem:I2-positive} (Lemma 2.3 of \cite{AMPVZ-K})}

First of all, $I$ has to be the Hardy operator and not just any integral operator.

By inspecting the proof of Lemma \ref{lem:I2-positive} (before we apply Cauchy-Schwarz we attach $\one_{ \text{supp}(f)}$ on the second term) to get eventually:

$$\|I ^*(I f\cdot g)\|_2^2\leq \int I \big(I ^*g \one_{ \text{supp}(f)}\big)\cdot (I f)^2g^2\leq \bigg(\sup\limits_{ \text{supp}(g)}I \big(I ^*g \one_{ \text{supp}(f)}\big)\bigg)\int (I f)^2g^2$$

Now let's look at $\sup\limits_{ \text{supp}(g)}I \big(I ^*g \one_{ \text{supp}(f)}\big)$ and  \text{supp}ose it is achieved for some $x\in T$. If $x\not\in  \text{supp}(f)$ then obviously $I \big(I ^*g \one_{ \text{supp}(f)}\big)(x)=I \big(I ^*g \one_{ \text{supp}(f)}\big)(\hat{x})$ where $\hat{x}$ is the father of $x$. If again $\hat{x}\not\in  \text{supp}(f)$ then we have again equality with the father of $\hat{x}$ and so on. We stop at the minimal $x'$ which is an ancestor of $x$ and
$x'\in  \text{supp}(f)$. So far we have equalities everywhere and therefore the supremum is achieved at some $x'\in  \text{supp}(f)$. Thus we have:

$$\sup\limits_{ \text{supp}(g)}I \big(I ^*g \one_{ \text{supp}(f)}\big)=\sup\limits_{ \text{supp}(g)\cap  \text{supp}(f)}I \big(I ^*g \one_{ \text{supp}(f)}\big)\leq \sup\limits_{ \text{supp}(g)\cap  \text{supp}(f)}I  I ^*g $$

Let $p\geq 2$ (although we care about $p>2$ as for $p=2$ the result has already been established) and $q$ its H\"older conjugate.

We work on simple finite trees $T_x, T_y$ and their product is a bi-tree which we denote by $\cT$. 

We say $g: T_x \rightarrow [0,+\infty)$ is of \textit{special form} if there is some $\beta\in T_y$ and a measure $m$ on $\cT$ such that $g(\gamma)=g^{\beta}(\gamma)=I_y\big(m^{q-1}(\gamma\times \boldsymbol{\cdot})\big)(\beta),\quad \forall \gamma\in T_x$ where $I_y$ is the one-dimensional Hardy operator on $T_y$.

\subsection{We want to prove the following theorem}

\begin{theorem}\label{inter2}
Let $f,g$ be positive functions on $T_x$ and numbers $0<\delta<<\lambda$ with $ \text{supp}(f)\subseteq\big\{ I  g\leq \delta\big\}$. If $g^{p-1}$ is super-additive, it has the above special form and $I g\leq \lambda$ on $T_x$ then we have for $p\geq 2$:

$$||I  f\cdot g||_{\ell^p(T_x)}\leq C \delta^\frac{p-1}{p}\cdot\lambda^{\frac{1}{p}} ||f||_{\ell^p(T_x)}$$
\end{theorem}

First we state two lemmata of \cite{AMPVZ-K}.

\begin{lemma}
\label{2.2}
Let $T'$ be a finite tree and $g,h : T' \to [0,\infty)$.
Assume that $g$ is super-additive and $I  h \leq \lambda$ on $\ \text{supp} g$.
Then for every $\gamma \in T'$ we have
\[
\sum_{\alpha \leq \gamma} g(\alpha)h(\alpha)
\leq
\lambda g(\gamma).
\]
\end{lemma}

and

\begin{lemma}
\label{2.3}
Let $I $ be the  on $T_x$ and $f,g$ positive functions.
Then
\[
\sum_{T_x} (I f)^{2} g \leq \Bigl( \sup_{\ \text{supp} g\cap \ \text{supp} f} I I ^{*}g \Bigr) \sum_{T_x} f^{2}.
\]
\end{lemma}

\begin{remark}
However, we failed to prove Theorem \ref{inter2}.  Here is an exhibition of what went wrong.
\end{remark}

\begin{proof}

First we use Lemma \ref{2.2} by setting $g\rightarrow g^{p-1}$ and $h\rightarrow g$. We have to check two things. First $I g \leq \lambda$ on $\ \text{supp} g^{p-1}$  holds as $ \text{supp} (g^{p-1})= \text{supp}(g)$ and we even have $I g \leq \lambda$ on the whole tree $T_x$. Second, we have to check $g^{p-1}$ is super-additive. 

Recall that  we have the following representation of $g$: 

$$g(\gamma)=g^{\beta}(\gamma)=\sum\limits_{\beta'\geq \beta}m^{q-1}(\gamma\times\beta'), \quad \forall \gamma \text{ in } T_x$$

Let $\gamma\in T_x$ be a dyadic interval and $\gamma_1, \gamma_2$ its two children.   

For the proof we make use of Minkowski's integral inequality for counting measures and exponent $p-1\geq 1$. We have

$$\Big(g^{p-1}(\gamma_1)+g^{p-1}(\gamma_2)\Big)^{\frac{1}{p-1}}= \Bigg(\sum\limits_{i=1}^2 \Big(\sum_{\beta'\geq \beta}m^{q-1}(\gamma_i\times\beta')\Big)^{p-1}\Bigg)^{\frac{1}{p-1}}\leq$$

$$\underset{Minkowski}{\leq}\sum_{\beta'\geq \beta} \Big(\sum\limits_{i=1}^2 m^{(q-1)(p-1)}(\gamma_i\times\beta')\Big)^{\frac{1}{p-1}}\underset{\substack{(q-1)(p-1)=1 \\ m \text{ is measure }}}{\leq}$$ 

$$\leq \sum_{\beta'\geq \beta} \Big( m(\gamma\times\beta')\Big)^{\frac{1}{p-1}}=g(\gamma)$$ as $\frac{1}{p-1}=q-1$.

Therefore, we conclude (by raising to the power $p-1$) that $g^{p-1}$ is super-additive.

What we achieved here is to prove for all $\gamma \in T_x$

\begin{equation}\label{gest}\sum_{\alpha \leq \gamma} g^p(\alpha)
\leq
\lambda g^{p-1}(\gamma) 
\end{equation}

This inequality would be useful and {\bf exactly what we would need to generalize results of} \cite{AMPVZ-K} {\bf from $p=2$ to $p\ge 2$}. But only if we could get a similar result to Lemma \ref{lem:I2-positive} (Lemma \ref{2.3} of \cite{AMPVZ-K}) with general $p\geq 2$ instead of $2$. Let's first state the result

\begin{lemma}
\label{new2.3}
Let $I $ be the Hardy operator on  a simple dyadic tree $T_x$ and $f,g$ positive functions with $g$ being increasing on $T_x$.
Then
\[
\sum_{T_x} (I f)^{p} g \leq \Bigl( \sup_{\ \text{supp} \,g\,\cap \ \text{supp} \,f} I I ^{*}g \Bigr) \sum_{T_x} f^{p}.
\]
\end{lemma}

We have added an extra assumption on $g$ that it has to be increasing. 

\begin{remark}
\label{mainremark}
 This is trues for $1\le p\le 2$ and false for $p>2$. We will see in section \ref{Counter} why it fails for  $p>2$.

But Theorem \ref{inter2} holds precisely for $p\ge 2$.  To prove the generalization  of  our embedding theorem for $p\neq 2$
we need both Theorem \ref{inter2} and Lemma \ref{new2.3} to be able to work together. But we claim that 
the first one works for $p\ge 2$ and the second one exactly for $1\le p\le 2$.
\end{remark}

\begin{remark}

If this theorem were true then we would use it with $g\rightarrow g^{p}$. We can see that since $g$ has the aforementioned special form, it is increasing, and so $g^p$ is also increasing. Therefore we get 

$$
\sum_{T_x} (I f)^{p} g^p \leq \Bigl( \sup_{\ \text{supp} g\cap \ \text{supp} f} I I ^{*}g^p \Bigr) \sum_{T_x} f^{p}.
$$
and by making use of \eqref{gest} (note that since $p-1\geq 1$ we have $I (g^{p-1})\leq (I g)^{p-1}$ ) and since $I g\leq \delta$ on $ \text{supp}(f)$ we get the desired estimate.

\end{remark}

\section{Relevant but not useful estimates in the positive direction}

Our main tool to prove Lemma \ref{new2.3} was to use Marcinkiewicz's interpolation. The proper way to do this is to check the boundedness of the operator $I $ from $\ell^{\infty}(T_x)$ to $\ell^{\infty}(T_x, g)$ for a particular fixed $g$. However this is a weird thing to ask as we ``almost'' have the equality $\ell^{\infty}(T_x)=\ell^{\infty}(T_x, g)$ where by almost we mean these are equal when $ \text{supp}(g)=T_x$. Also, $\|I  f\|_{\ell^{\infty}(T_x, g)}= \|I  f\|_{\ell^{\infty}(T_x)}$ except possibly the case where the supremum on RHS is achieved at a node in $T_x\setminus  \text{supp}(g)$.  Thus we can not interpolate, as its not possible to get this estimate.

Now let's see what happens if we interpolate in other ways.

By replacing $g$ with $g^2$ in Lemma \ref{2.3} we have seen that 

$$\int (If)^{2} g^2 \leq \Bigl( \sup_{\ \text{supp} g} II^{*}g^2 \Bigr) \int f^{2}$$ 

Let $A_{g^2}:=\sup\limits_{\ \text{supp} g} II^{*}g^2$. The above becomes

$$\|If\cdot g\|_{\ell^2}\leq A_{g^2}^{\frac{1}{2}}\|f\|_{\ell^2}$$

We have already proven (see Lemma \ref{infinity}) that for $g$ increasing:

$$\|If\cdot g\|_{\ell^{\infty}}\leq A_{g}\|f\|_{\ell^{\infty}}$$

where $A_{g}:=\sup\limits_{\ \text{supp} g} II^{*}g$. We can interpolate then, to get

$$\|If\cdot g\|_{\ell^p} \leq A_{g^2}^{\frac{1}{p}}\cdot A_g^{1-\frac{2}{p}}\|f\|_{\ell^p}$$

while if we interpolate between $\ell^1$ and $\ell^{\infty}$ we get

$$\|If\cdot g\|_{\ell^p} \leq  A_g\|f\|_{\ell^p}$$

which are not useful as $g$ is sub-additive.

\section{Yet another try; We follow the proof of Lemma \ref{lem:I2-positive} (Lemma 2.3 in  \cite{AMPVZ-K})}

In this case we see that 

$$\sum_T (If)^pg^p \leq \|f\|_{\ell^p}\cdot \|I^*\big((If)^{p-1}g^p\big)\|_{\ell^q}$$
and using Young's (product of numbers) inequality we get for the latter term

$$\|I^*\big((If)^{p-1}g^p\big)\|_{\ell^q}^q\leq \frac{1}{p}\Bigl( \sup_{\ \text{supp} g} II^{*}(g^p) \Bigr) \sum_T (If)^pg^q + \frac{1}{q}\Bigl( \sup_{\ \text{supp} g} II^{*}(g^q) \Bigr) \sum_T (If)^pg^p$$
which is exactly the same as Lemma 2.3 when $p=q=2$, but again to estimate $II^*g^q$ we need information about $g^{q-1}$ and this function is sub-additive (by inverse Minkowski).

\end{proof}

\section{Counterexample to Lemma \ref{new2.3}}\label{Counter}

All this is a try to prove Lemma \ref{new2.3} (of this note), but as we will see this is in vain. We will construct a counterexample to this Lemma for $p>2$ (for $1\leq p \leq 2$ the theorem is true by interpolation).

We will give a counterexample for a super-additive function $g$ with this property:  $\|II^*g\|_{\infty}$ has to be achieved for some $\om\in \pd T$ with $ \om\in  \text{supp}(g)$. The theorem we wanted to prove conjectured that for $g$ increasing would be sufficient, but with this counterexample we will show that even super-additivity is not enough.
Also we will create a sub-additive function with this property, hence disproving the Lemma in case we wanted to replace ``increasing'' with ``sub-additive''.

Fix an $\om$ with this property. Then we have $\|II^*g\|_{\infty}=II^*g(\om)$. See below about a possible choice of this function.

We construct the function $f$ as follows. It is equivalent to $1$ inside the totally ordered set $\{\beta: \beta\geq \om \}$ and $0$ otherwise. We enumerate the elements of this set: $u_1=I_0$,..,$u_N=\om$. By construction of $f$ and since $u_k$ lives in the $k$-th generation with $1\leq k \leq N$, we have $If(u_k)=k$ and hence :

$$\sum_T (If)^p g \geq \sum\limits_{k=1}^N (If)^p(u_k) g(u_k)\geq \sum\limits_{k=2}^N k^p g(u_k)$$

Keep in mind that $g(u_k)=I^*g(u_k)-I^*g(u_{k-1})$ and so 

$$\sum\limits_{k=2}^N k^p g(u_k)=\sum\limits_{k=2}^N I^*g(u_k)\big(k^p-(k-1)^p\big)\geq \sum\limits_{k=2}^N I^*g(u_k)\big(k^p-(k-1)^p\big)\geq $$

$$\geq\frac{p}{2^{p-1}}\sum\limits_{k=2}^N I^*g(u_k)k^{p-1}$$ as $\frac{k}{2}\geq 1$.

Also, note that $I^*g(u_k)=II^*g(u_k)-II^*g(u_{k-1})$ and thus

$$\sum\limits_{k=2}^N I^*g(u_k)k^{p-1}=\sum\limits_{k=2}^N  \big(II^*g(u_k)-II^*g(u_{k-1})\big)k^{p-1}=$$

$$II^*g(u_N)N^{p-1} -II^*g(u_1)1^{p-1} -\sum\limits_{k=1}^{N-1}II^*g(u_k)\big((k+1)^{p-1}-k^{p-1}\big) \quad (\star)$$ 

Recall that $u_N=\om$ and $u_1=I_0$ and by assumption $II^*g(\om)=\|II^*g\|_{\infty}$ and therefore

$$(\star)\geq \|II^*g\|_{\infty}\cdot N^{p-1}  - II^*g(I_0) - \|II^*g\|_{\infty}\cdot \sum\limits_{k=m}^{N-1}\big((k+1)^{p-1}-k^{p-1}\big)\geq$$

$$\|II^*g\|_{\infty}\cdot N^{p-1}  - II^*g(I_0) - \|II^*g\|_{\infty}\cdot (p-1)\cdot\sum\limits_{k=1}^{N-1}k^{p-2}\quad (\star\star)$$

and since $\sum\limits_{k=1}^{N-1}k^{p-2} \leq \int\limits_{1}^{N-1}x^{p-2}{d}x = \frac{1}{p-1}\cdot\big((N-1)^{p-1}- 1^{p-1}\big)$ we have

$$(\star\star) \geq \|II^*g\|_{\infty}\cdot N^{p-1}  - II^*g(I_0)  - \|II^*g\|_{\infty}\cdot\big((N-1)^{p-1}- 1\big)=$$

$$\geq \|II^*g\|_{\infty}\cdot \big( N^{p-1}-(N-1)^{p-1}\big)$$

By keeping track of one constant which we left behind, we conclude:

$$\sum_T (If)^p g \geq \frac{p}{2^{p-1}}\cdot \|II^*g\|_{\infty}\cdot \big( N^{p-1}-(N-1)^{p-1}\big)$$

Now \textbf{if} the theorem was true, we would also have

$$\sum_T (If)^p g \leq \|II^*g\|_{\infty} \sum_T f^p$$ but the RHS is equal to $\|II^*g\|_{\infty}\cdot N$ as $f$ is equal to $1$ exactly on $N$ nodes. Therefore 

$$\frac{p}{2^{p-1}}\cdot \|II^*g\|_{\infty}\cdot \big( N^{p-1}-(N-1)^{p-1}\big) \leq \|II^*g\|_{\infty}\cdot N $$

We can cancel the common term and by dividing with $(N-1)^{p-1}$ and taking the limit as $N$ goes to infinity we get :

i) the RHS tends to $0$ (as $p-1>1$) and  

ii) the LHS tends to to $\frac{p}{2^{p-1}}\big(e^{p-1}-1\big)$

Therefore we have a contradiction.

\subsection{Creating $g$ functions} Take $N$ sufficiently large such that the above limiting argument gives a contradiction (everything over there is independent of functions, its just calculus 0).

Then we construct a super-additive $g$  on $T=T_N$ which is a finite simple tree of depth $N$.

We beging by letting $g(I_0)=1/2$. Then we put $g(I_0^+)=1/4,g(I_0^-)=0$. In general: on the left child of every dyadic interval where $g$ is $a\neq 0$ we put the value $a/2$ and on the right we put $0$. We also put $0$ on both children of a dyadic interval where $g$ is 0. Obviously such $g$ is super-additive. 

In general $II^*g(\al)\leq II^*g(\om)$ for $\al> \om$. Now if we  choose $\om \in \pd T$ such that $g(\om)=2^{-N}$ (there is exactly one such $\om$, far-left on the bottom of the tree) we have a strict inequality (i.e $II^*g(\al)< II^*g(\om)$ ) and therefore the maximum is achieved on the boundary. Hence this function satisfies the basic requirement. 

By going back to Lemma \ref{new2.3}, replacing $g$ with $g^p$ and ``increasing'' with ``$g^{p-1}$ super-additive'' can not be proven either. For example, take this same function $g$ as above. Note it is such that $g^{p-1}$ is super-additive and $\|II^*g^p\|_{\infty}$ is achieved on the boundary.

By taking a function $g$ which is equivalent to $1$ on the whole tree, we get a sub-additive function. Note that $g^{p-1}\equiv g.$ Again for this function  $\|II^*g\|_{\infty}$ is achieved on the boundary and hence this is a counterexample to Lemma \ref{new2.3} if we would replace increasing with sub-additive.

\section{A counterexample to Lemma \ref{Phi} on bi-tree}
\label{fg-cex}

\subsection{Statement of the problem}

Let $\bI$ be operator of summation ``up the graph''. It has a formally adjoint operator $\bI^*$ of summation ``down the graph''. We use the same notation for the rooted dyadic tree $T$ and for $T^2$.
On dyadic tree $T$ we have the following key ``majorization theorem with small energy'':
\begin{theorem}
\label{d1}
Let $f, g: T\to \bR_+$, and  1) $f$ is superadditive, 2)  $\ \text{supp} f \subset \{\bI g \le \delta\}$.  Let $\la \ge 10\delta$. Then there exists
$\vf: T\to \bR_+$  such that
\begin{enumerate}
\item $\bI \vf \ge \bI f$ on $\{2\la \le \bI g \le 4\la\}$;
\item $\int_T \vf^2 \le C\frac{\delta^2}{\la^2} \int_T f^2$.
\end{enumerate}
\end{theorem}

For a while we tried to prove the similar statement for $T^2$. Namely, we conjectured 
\begin{conj}
\label{d2}
Let $f, g: T^2\to \bR_+$, and  1) $f$ is superadditive in each variable, 2)  $\ \text{supp} f \subset \{\bI g \le \delta\}$.  Let $\la \ge 10\delta$. Then there exists
$\vf: T^2\to \bR_+$  such that
\begin{enumerate}
\item $\bI \vf \ge \bI f$ on $\{2\la \le \bI g \le 4\la\}$;
\item $\int_{T^2} \vf^2 \le C\frac{\delta}{\la} \int_{T^2} f^2$.
\end{enumerate}
\end{conj}

For some very special cases, e. g. for $f=g$, this has been proved, and turned out to be a key result in describing the embedding  measures for the Dirichlet spaces in tri--disc into $L^2(\bD^3, d\rho)$. See
\cite{AMPVZ-K}, \cite{MPVZ1}, \cite{MPVZ2}

\bigskip

Now we will show that this is not true in general. 

Moreover, below $f, g$ have special form, namely
$$
f=\bI^*\mu, \,g=\bI^*\nu,
$$
 with certain positive measures on $T^2$.  And measure $\mu$ is trivial, it is a delta measure of  mass $1$ at the root $o$ of $T^2$. In particular, $f(o)=1, f(v)=0, \forall  v\neq o$. Also $\bI f \equiv 1$ on $T_2$.

\bigskip

The choice of $\nu$ is more sophisticated. Choose large $n= 2^s$ and denote $2^M:= \frac{n}{\log n}$.

In the unit square $Q^0$ consider dyadic sub-squares $Q_1, \dots, Q_{2^M}$, which are South-West to North-East diagonal squares of size $2^{-M}\times 2^{-M}$.

In each $Q_j$ choose $\om_j$, the South-West corner dyadic square of size $2^{-n} \cdot 2^{-M}$.

Measure $\nu$ is the sum of delta measures at $\om_j, j=1, \dots, \frac{n}{\log n}$, each of muss $\frac1{n^2}$. Obviously 
$$
g(o)=\bI^* \nu(o)=\bI\bI^* \nu(o)= (\bV^* \nu)(o) = \|\nu\|= \frac1{n^2}\cdot \frac{n}{\log n}= \frac1{n\log n}=:\delta.
$$
So we chose $\delta$ and $f, g$ satisfy $\ \text{supp} f =\{o\}\subset \{ \bI g \le \delta\}$. Also $f$ is sub-additive in both variables on $T^2$: it is just a characteristic function of the root.

\medskip

Now what is $\la$, and what is the set $\{2\la \le \bI g \le 4\la\}$?

\medskip

Consider (by symmetry this will be enough) $Q_1$ and $\om_1$ and consider the family $\cF_1$ of dyadic rectangles containing $\om_1$ and contained in $Q_1$ of the following sort:
$$
[0, 2^{-n} 2^{-M}]\times [0, 2^{-M}], [0, 2^{-n/2} 2^{-M}]\times [0, 2^{-2}2^{-M}], \dots, 
$$
$$
[0, 2^{-n/2^k} 2^{-M}]\times [0, 2^{-2^k}2^{-M}],\dots,
$$
there are approximately $\log n$ of them, and they are called $q_{10}, q_{11},\dots, q_{1k},\dots$.

\begin{lemma}
\label{g}
$\bI g(q_{ik}) \asymp \frac1{n}\quad \forall k$.
\end{lemma}

It is proved in \cite{AMPVZ-K}, \cite{MPV}.

\bigskip

Let $F:=\cup_{ik} q_{ik}$.

So we choose $\la=\frac{c}n$ with appropriate $c$.  Then
$$
F\subset \{2\la \le \bI g \le 4\la\}\,.
$$

As it is obvious that $\bI f\ge 1$ everywhere, so if $\vf$ as in Conjecture \ref{d2} would exist, we would have $\bI \vf \ge 1$ on $F$ and  (by the second claim of Conjecture \ref{d2})
$$
\int_{T^2} \vf^2 \le  \frac{C\delta}{\la} \int_{T^2} f^2=  \frac{C}{\log n} \int_{T^2} f^2 = \frac{C}{\log n}\,.
$$
By the definition of capacity this would mean that
$$
\text{cap} (F) \le \frac{C}{\log n}\,.
$$

In the next section we show that $\text{cap} (F) \asymp 1$. Hence, Conjecture \ref{d2} is false.

\subsection{Capacity of $F$ is equivalent to $1$}

Let $\rho$ on $F$ be a capacitary measure of $F$, and $\rho_{jk}$ be its mass on $q_{jk}$.
By symmetry $\rho_{jk}$ does not depend on $j=1, \dots, \frac{n}{\log n}$.

\medskip 

The proof of the fact that 
$$
\rho_k:= \rho_{jk},\,\, j=1, \dots, \frac{n}{\log n}
$$
have the average $\ge \frac{c_0}{n}$, that is that
\begin{equation}
\label{E}
\bE \rho = \frac{\sum_{k=\frac{\log n}{2}}^{\frac{3\log n}{4}} \rho_k}{1/2\log n} \ge \frac{c_0}{n}
\end{equation}
follows below.

\bigskip  

In its turn it gives the required
\begin{equation}
\label{capF}
\text{cap} F \asymp 1\,.
\end{equation}

\medskip

Let us  first derive \eqref{capF} from \eqref{E}.
Measure $\mu$ that charges $\rho_k$ on each $q_{jk}, j=1,\dots, \frac{n}{\log n}; k= \frac{\log n}{2},\dots, \frac{3\log n}{4}$ is equilibrium so it gives $\bV^\mu \equiv 1$ on each $q_{jk}$.
Then \eqref{capF} follows like this: $\text{cap} F= \|\mu\|= \frac{n}{\log n} \sum_{k=\frac{\log n}{2}}^{\frac{3\log n}{4}} \rho_k = \frac12 n \bE \rho$. Hence $\text{cap} F= \|\mu\|\ge \frac12 c_0$ if \eqref{E} is proved.

\medskip

Now let us prove \eqref{E}.
Everything is symmetric in $j$, so let $j=1$ and let us fix $k$ in $[\frac{\log n}{2}, \frac{3\log n}{4}]$. We know that
$$
1\ge \bV^\mu \quad \text{on} \,\, q_{1k},
$$
and now let us estimate this potential from above.  For that we split $\bV^\mu$ to
$\bV_1$, this is the contribution of rectangles  containing $Q_1$, to
$\bV_2$, the contribution of rectangles containing $q_{1k}$ and contained in $Q_1$, and
$\bV_3$, the contribution of rectangles containing $q_{1k}$ that strictly intersect $Q_1$ and that are ``vertical'', meaning that there vertical side contains vertical side of $Q_1$.
(There is $\bV_4$ totally symmetric to $\bV_3$.)

Two of those are easy, $\bV_1$ ``almost'' consists of ``diagonal squares containing $Q_1$. Not quite, but other rectangles are also easy to take care. 
Denote
$$
r=\|\mu\|, \quad M=\log\frac{n}{\log n}\,.
$$
Then we write diagonal part first and then the rest:
$$
\bV_1= r+\frac{r}{2} +\frac{r}{4} + \dots \frac{r}{2^M}  +\frac{r}{2}  +\frac{r}{2} + 2\frac{r}{4} + 2\frac{r}{4} +
\dots   k\frac{r}{2^k}  + 2\frac{r}{2^k}  +\dots = C_1 r
$$

To estimate $\bV_2$ notice that there are at most $c n$ rectangles containing $q_{1k}$ and contained in $Q_1$ that do not contain any other $q$, there are $\frac{c n}{2}$ of rectangles contain $q_{1k}$ and one of its sibling (and lie in $Q_1$),  there are $\frac{c n}{4}$ of rectangles contain $q_{1k}$ and two of its sibling (and lie in $Q_1$), et cetera.

Hence,
$$
\bV_2 \le Cn\rho_k + \frac{Cn}{2} \rho_{k\pm 1} +  \frac{Cn}{4} \rho_{k\pm 2}+\dots
$$

Now consider $\bV_3$. The horizontal size of $q_{1k}$  is $2^{-M}\cdot 2^{-n2^k}$. Its vertical size
is $2^{-M}\cdot 2^{-2^k}$. So the rectangles of the third type that do not contain the siblings: their number is
at most (we are using that $k\ge \frac12 \log n$)
$$
n2^{-k} ( 2^k+M) \le n + \sqrt{n}\log n\,.
$$
Those that contain $q_{1k}$ and one sibling, there number is at most
$$ 
n2^{-k} ( 2^{k-1}+M) \le \frac{n}2 + \sqrt{n}\log n\,.
$$
We continue, and get that
$$
\bV_3 \le n \rho_k + \frac{n}2  \rho_{k\pm1} + \frac{n}4 \rho_{k\pm 2}+\dots + \sqrt{n}\log n (\sum \rho_s)\,.
$$
Add all $\bV_i$:
$$
1 \le \bV_1+\bV_2 +\bV_3 +\bV_4 \le  C_1 r+ C n \rho_k+ \frac{Cn}2  \rho_{k\pm1} + \frac{Cn}4 \rho_{k\pm 2}+\dots + \sqrt{n}\log n (\sum \rho_s)\,.
$$
Now average over $k$. Notice that 
$$
r=\|\mu\|= \frac{n}{\log n} \sum\rho s  =\frac12  n \bE \rho
$$
Hence,
$$
1 \le C' n \bE\rho + Cn \bE\rho + \frac{Cn}2 \bE\rho+  \frac{Cn}4 \bE\rho + \dots +  \frac12\sqrt{n}\log^2 n \,\bE\rho\,.
$$
Therefore, $\bE\rho\ge \frac{c_0}{n}$ and \eqref{E} is proved.

\end{document}